\documentclass{amsart}

\title{Complete $\lomegaone$-Sentences with Maximal Models in Multiple Cardinalities}
\author[J. Baldwin]{John Baldwin }
\address{Department of Mathematics, Statistics, and Computer Science, University of Illinois at Chicago, 851 S. Morgan
St. (M/C 249), Chicago, IL 60607-7045, USA}
\email{jbaldwin@uic.edu}


\author[I. Souldatos]{Ioannis  Souldatos}
\address{4001 W.McNichols, Mathematics Department, University of Detroit Mercy, Detroit, MI 48221, USA}
\email{souldaio@udmercy.edu}

\date{\today}
\thanks{Funding: Research partially supported by Simons travel grants G5402  and 418609.}
\usepackage[latin1]
{inputenc}

\usepackage{geometry,amssymb,amsthm,amsmath,
graphics,enumerate}
\usepackage{enumerate}
\usepackage{url}
\usepackage{hyperref}

\theoremstyle{plain}
\newtheorem{Thm}{Theorem}[subsection]

\newtheorem{question}[Thm]{Open Question}
\newtheorem{Cor}[Thm]{Corollary}

\newtheorem{Def}[Thm]{Definition}
\newtheorem{Rem}[Thm]{Remark}

\newtheorem{Not}[Thm]{Notation}
\newtheorem{Fact}[Thm]{Fact}

\newcommand{\M}{\ensuremath{\mathcal{M}}}
\newcommand{\N}{\ensuremath{\mathcal{N}}}

\newcommand{\cf}{\mathrm{cf}}

\newcommand{\lomegaone}{\ensuremath{\mathcal{L}_{\omega_1,\omega}}}

\def\hbar{{\bf h}}

\def\sbar{{\bf s}}

\begin{document}

\begin{abstract} In \cite{BKSoul} examples of incomplete sentences are given with maximal
models in more than one
cardinality. The question was raised whether one can find similar examples of complete sentences.
 In this paper we give
 examples of complete $\lomegaone$-sentences with maximal models in more than one
 cardinality.  From (homogeneous) characterizability of $\kappa$ we construct sentences with
 maximal models in $\kappa$ and  in one of $\kappa^+, \kappa^\omega,
 2^\kappa$ and more.
 Indeed, consistently we find sentences with maximal models in uncountably many distinct cardinalities.

\end{abstract}
\maketitle

We unite ideas from \cite{BFKL, BKL, Hjorthchar,Knightex} to find
complete sentences of $\lomegaone$ with maximal models in multiple
cardinals.  There have been a number of papers finding complete
sentences characterizing cardinals beginning with Baumgartner, Malitz
and Knight in the 70's, refined by Laskowski and Shelah in the 90's
and crowned by Hjorth's characterization of all cardinals below
$\aleph_{\omega_1}$ in the 2002.  These results have been refined
since. But this is the first paper finding complete sentences with
maximal models in two or more cardinals. All models of these sentences have cardinality less than $\beth_{\omega_1}$.

Our arguments combine and extend the techniques of building atomic
models by Fraiss\'{e} constructions using disjoint amalgamation,
pioneered by Laskowski-Shelah and Hjorth, with the notion of
homogeneous characterization  and tools from
Baldwin-Koerwien-Laskowski (\cite{BKL}). This paper uses specific techniques from \cite{BFKL, BKL,
SouldatosCharacterizableCardinals, Souldatoscharpow} and many proofs
are adapted from these sources.  

\textbf{Structure of the paper}:

In Section \ref{template}, we explain the merger techniques for
combining sentences that homogeneously characterize one cardinal
(possibly in terms of another).  We  adapt the methods of
\cite{SouldatosCharacterizableCardinals}
to get a  complete sentence with maximal models in $\kappa$ and
$\kappa^+$.

In Section \ref{sec:ltoomega} we present, for each homogeneously characterizable $\kappa$,
 an $\lomegaone$-sentence with
maximal models in $\kappa$ and $\kappa^\omega$ and no larger models.
The argument can be generalized to obtain maximal models in $\kappa$
and $\kappa^{\aleph_\alpha}$, for all countable $\alpha$.

Finally in Section \ref{ManyCardinalities}, we give various examples
of complete sentences with maximal models in a number of
cardinalities, modulo appropriate hypotheses on cardinal arithmetic.
For example, Corollary~\ref{maxmodels} asserts that  if
$\kappa$ is homogeneously characterizable and $\mu$ is minimal with
$2^{\mu} \geq \kappa$  there is an
$\lomegaone$-sentence $\phi_\kappa$ with maximal models in cardinalities
$2^{\lambda}$ for $\mu \leq \lambda \leq \kappa$ and no models larger
than $2^{\kappa}$.

\section{The general construction}\label{template}
\numberwithin{Thm}{section} \setcounter{Thm}{0}

In this section, for a cardinal $\kappa$ that admits a homogeneous
characterization (Definition\ref{Def:HomChar}), we prove that there
exists a complete sentence $\phi_\kappa$ of $\lomegaone$ that has
maximal models in $\kappa$ and $\kappa^+$ and no larger models. The
proof applies the notion of a receptive model from \cite{BFKL} and
merges a sentence homogeneously characterizing $\kappa$ with a
complete sentence encoding uniformly the transfer from characterizing
$\kappa$ to characterizing $\kappa^+$ by a $\kappa^+$-like linear
order. This template is extended from successor to other cardinals functions in later sections.

We require a few preliminary definitions.

\begin{Def}\label{Def:HomChar} Assume $\lambda\le\kappa$ are infinite cardinals,
$\phi$ is a complete $\lomegaone$-sentence in a vocabulary that
contains a unary predicate $P$, and $\M$ is the (unique) countable model of $\phi$. We say
\begin{enumerate}
\item  a
    model $\N$ of $\phi$ is of type $(\kappa,\lambda)$, if
    $|\N|=\kappa$ and $|P^\N|=\lambda$;
    \item For a countable structure $\M$, $P^\M$ is a set of {\em
        absolute indiscernibles} for $\M$, if $P^\M$ is infinite
        and every permutation of $P^\M$ extends to an
        automorphism of $\M$.
 \item $\phi$ \emph{homogeneously characterizes} $\kappa$, if
\begin{enumerate}
  \item $\phi$ has no model of size $\kappa^+$; \item $P^\M$
      is a set of absolute indiscernibles for the countable
      $\M$, and
 \item there is a maximal model of $\phi$ of type $(\kappa,\kappa)$.
\end{enumerate}
\end{enumerate}
\end{Def}

The next notation is  useful for defining mergers. We slightly
broaden the notion of `receptive' from \cite{BFKL} by requiring some
sorts of the `guest sentence' to restrict to $U$ while others include
new sorts in the final vocabulary.

\begin{Not}
\label{fixvoc1} Fix a vocabulary $\tau$ containing  unary predicates
$V,U$.  The sentence $\theta_0$ says
$V$ and $U$ partition the universe.

Let $\tau_1$ extend $\tau$ and let $\theta$ be a complete
$\tau_1$-sentence of $\lomegaone$ that implies $\theta_0$. Fix a
vocabulary $\tau'$  disjoint from $\tau_1$ that contains a unary
predicate $Q$, and let $\psi$ an arbitrary (possibly incomplete)
$\tau'$-sentence of $\lomegaone$. Let $\tau_2$ contain the symbols of
$\tau_1 \cup \tau'$, adding unary predicates $R$ and $S$.

The formula $\theta_1$ asserts that the $\tau'$-predicates hold only
of tuples from $R$ and the $\tau_1$-predicates only of tuples from
$S$, that $(S(x)\wedge R(x)) \leftrightarrow U(x)$, and that $U(x)
\leftrightarrow Q(x)$.
\begin{itemize}
\item If $U$ defines an infinite {\em absolutely indiscernible
    set} in the countable model of $\theta$, we call the pair
    $(\theta,U)$ \emph{receptive}. We call $\theta$
    \emph{receptive} if there is a $U$ such that $(\theta,U)$ is
    receptive and in that case we also call the countable model
    of $\theta$ a receptive model.

\item The \emph{merger} $\chi_{\theta,U,\psi,Q}$ of the pair
    $(\theta,U)$ and $(\psi,Q)$ is the conjunction of
    $\theta_1$ and $\psi^{U,Q}$, where $\psi^{U,Q}$
    is the result of interpreting $\theta$ on $S$,
    substituting $U$ for $Q$ in $\psi$ and interpreting $\psi$ on
    $R$. Thus $\chi_{\theta,U,\psi,Q}$ is a $\tau_2$-sentence.
    \item If in all models $\N$ of $\psi$, $Q^\N$ is the domain
        of $\N$, then we will drop $Q$ and write
        $\chi_{\theta,U,\psi}$.
    \item If $\M\models\theta$ and $\N\models\psi$, the
        \emph{merger model} $(\M,\N)$ denotes a model of
        $\chi_{\theta,U,\psi,Q}$ where the elements of $Q^\N$
        have been identified with the elements of $U^\M$, which
        is the intersection of $M$ and $N$.

$\M$ will be called the
\emph{host} model and $\N$ the \emph{guest
model}.
\end{itemize}
\end{Not}

Note that if $\phi$ and $P$ homogeneously characterize some $\kappa$,
then the countable model of $\phi$ is receptive. Fact~\ref{mergeprop}
extends the argument for Theorem 1.10 in \cite{BFKL} to reflect our more general
notion of merger.

\begin{Fact}\label{mergeprop}  Let $(\theta,U)$ be receptive
and $\psi$ a sentence of $\lomegaone$.
\begin{enumerate}
\item The merger $\chi_{\theta,U,\psi,Q}$ is a complete sentence
   if and only if $\psi$ is complete.
\item There is a $1$-$1$ isomorphism preserving function between isomorphism types of
   the countable  models of $\psi$ and the isomorphism types of countable models of the  merger
   $\chi_{\theta,U,\psi}$.

\item If there is a model $M_0$ of $\theta$ with
   $|M_0|=\lambda_0$ and   $|U^{M_0}|=\rho$  and also a model
   $M_1$ of $\psi$ with $|M_1|=\lambda_1$ and $|Q^{M_1}|=\rho$,
   then there is a model of $\chi_{\theta,U,\psi,Q}$ with
   cardinality $\max(\lambda_0,\lambda_1)$.
\end{enumerate}
\end{Fact}

\begin{Rem}\label{getcomp} {\rm The proof of 1) of Fact 1.10 in \cite{BFKL} is a bit quick.  The completeness
also depends on absolute indiscernability. Let $N$ and $N'$ be
countable models of  $\chi_{\theta,U,\psi,Q}$. Let $\alpha$ be a
bijection between $N$ and $N'$ which is a  $\tau_1$-isomorphism of
$S(N)\restriction \tau_1$ onto $S(N')\restriction \tau_1$.
Push-through the $\tau'$-structure on $R(N)$ to $R(N')$ to give a
structure $N''$ such that for $\sbar \in R(N)$, $P\in \tau'$, $N''\models
P(\alpha(\sbar))$ if and only if $N\models P(\sbar)$; so $R(N'')
\restriction \tau' \models \psi$.

 Let $\gamma$ be a permutation of
$R(N'')= R(N')$ that is an isomorphism from  the $\tau'$-structure
imposed on $R(N')$ by $\alpha$ to $R(N') \restriction \tau'$ (and so
fixes $U$ setwise). Now by absolute indiscernability, extend
$\gamma\restriction U$ to an automorphism of $N'$.  Then $\gamma
\circ\alpha$ is a $\tau_2$-isomorphism between $N$ and $N'$ as
required.}
\end{Rem}

Using this result we show if $\kappa$ is homogeneously
characterizable, we can construct a complete sentence of $\lomegaone$
that has maximal models in $\kappa$ and $\kappa^+$ and no larger
models. Before we proceed
 with the proof we introduce the tool by which we turn homogeneously
 characterizable cardinals into pairs of maximal models.

\begin{Thm}\label{Thm:ManyMaximal} Let $\theta$ homogeneously characterize $\kappa$. Then there exists an
$\lomegaone$-sentence $\chi= \chi_\theta$ in a vocabulary with a new
unary predicate symbol $B$, such that $(\chi,B)$ is receptive, $\chi$
homogeneously characterizes $\kappa$ and $\chi$ has maximal models of type $(|\M|, |B^{\M}|)= (\kappa,\lambda)$, for all
$\lambda\le\kappa$.
 \end{Thm}

\begin{proof}
 Fix a receptive pair $(\theta,U)$ such that  $\theta$  homogeneously
characterizes $\kappa$.  Define a new vocabulary $\tau=\{A,B,p\}$
where $A,B$ are unary predicates and $p$ is a binary predicate. Let
$\phi_0$ be the conjunction of: (a) $A,B$ partition the universe and
(b) $p$ is a total function from $A$ onto $B$ such that each
$p^{-1}(x)$ is infinite. By Theorem 1.10 of \cite{BFKL}  there is a
complete sentence $\phi$ that implies $\phi_0$ and  in the countable
model of $\phi$, $B$ is a set of absolute indiscernibles.

Now merge $\theta$ and $\phi$ by identifying $U$ and $A$. The merger
$\chi=\chi_{\theta,U,\phi,A}$ is a complete sentence which does not have
any models of size $\kappa^+$. Let $\M$ be a maximal model of
$\theta$ with $U^\M$ of size $\kappa$, and $\N$ a model of $\phi$ of
type $(\kappa,\lambda)$, for some $\lambda\le\kappa$. Then the merger
model $(\M,\N)$ is a maximal model of $\chi$ with $|(\M,\N)|=\kappa$ and $|B^{(\M,\N)}|=
\lambda$, which proves the result
\end{proof}

A word of caution: In the countable model of $\theta$, the predicate
$U$ defines a set of absolute indiscernibles in the countable model,
and the same is true for the countable model of $\phi$ and  $B$. So,
we started with two models and two sets of absolute indiscernibles.
In the merger $\chi_{\theta,U,\phi,A}$, the absolute indiscernibles
of the host model (model of $\theta$) are used to bound the size of
$A$ from the guest model (model of $\phi$). Moreover, the predicate
$B$ from the guest model defines a set of absolute indiscernibles in
the merger model too.

The construction in the next theorem  extends by the use of
Theorem~\ref{Thm:ManyMaximal} the 50 year old  argument that if
$\kappa$ is characterized then so is $\kappa^+$ to obtain maximal
models in distinct cardinalities.

\begin{Thm}\label{getcompnew}
Suppose $\theta$ is a complete sentence of $\lomegaone$ that homogeneously characterizes $\kappa$. Then there is
 a complete sentence  $\psi= \psi_\theta$ of $\lomegaone$ such that
 $\psi$ characterizes $\kappa^+$ and has maximal models in $\kappa$
 and $\kappa^+$.
\end{Thm}

Proof. We can replace the $\theta$ homogeneously characterizing $\kappa$ by
the $\tau =\tau_{\chi_\theta}$-sentence $\chi_\theta$ from
Theorem\ref{Thm:ManyMaximal} that homogeneously characterizes
$\kappa$ with set of absolute indiscernibles $B$ and which has maximal models of type
$(\kappa,\lambda)$ for each $\lambda \leq \kappa$.

Let $\psi =\psi_\theta$ be the conjunction of the following
sentences, in a vocabulary $\tau'$ that contains  unary predicates
$Q_1,Q_2$,  binary predicates $<, P$,  and for each $k$-ary $R \in
\tau$ a $k+1$-ary predicate $\hat R$ in $\tau'$. The axioms assert:

\begin{enumerate}
\item $Q_1,Q_2$ partition the universe.
\item  $(Q_2,<)$ is a dense linear order without endpoints.

\item $P$ is the graph of a function from $Q_1$ to $Q_2$.
\item For every predicate $R(\vec{x}$) in $\tau$, if $\hat
    R(a,\vec{x})$ in $\tau'$ holds, then all members of $\vec{x}$
    belong to $D_a=P^{-1}(a)\cup \{y\in Q_2|y<a\}$.

 \item For every $a$ in $Q_2$, the set $D_a$ with the
     $\tau$-structure obtained by interpreting $\hat
     R(a,\vec{x})$ as $R(\vec{x})$
    and $\{y\in Q_2|y<a\}$ as $B^{D_a}$  is isomorphic to a model
    of $\theta$.
\end{enumerate}

Note that  for any $a$ and $\hat R$, $\hat
    R(a,\vec{c})$ holds for a vector $\vec{c}$
of distinct elements of $\{y:y<a\}$ if and only if it holds for all
such tuples (by the absolute indiscernability of $B$ in models of
$\chi_\theta$).

 We prove any two countable models, $M, N$ of $\psi$ are
isomorphic.  Fix an isomorphism $\alpha$ from $(Q_2^M,<^M)$ onto
$(Q_2^N,<^N)$.  As in Remark~\ref{getcomp} we now extend the $\alpha
\restriction \{y:y<^Ma\}$ to a family of
$\tau'$-isomorphisms $\alpha_a$ between $M \restriction D^M_a$ and $N
\restriction D^N_{\alpha(a)}$. By the categoricity of $\theta$, there exists a $\tau$-isomorphism $\rho$ between $M \restriction
D^M_a$ and $N
\restriction D^N_{\alpha(a)}$ (and $\rho$ induces a
$\tau'$-isomorphism). But we don't know {\em a priori} that $\rho
\restriction \{y:y<^Ma\}= \alpha \restriction \{y:y<^Ma\}$.
 Let $\gamma$ be a permutation of
$\{y:y<^N\alpha(a)\}$ that is an  order isomorphism between the order
given by $\rho$ and the one imposed by $\alpha$. Now extend
$\gamma$ by absolute indiscernability to an automorphism of
$D^N_{\alpha(a)}$. Then $\alpha_a=\gamma \circ\rho$ is a $\tau'$-isomorphism
between $D^M_a$ and $D^N_{\alpha(a)}$ that extends $\alpha\restriction \{y:y<^Ma\}$. For $b<a$, $\alpha_a$ and $\alpha_b$
agree on their common domain, since their domains intersect only
on $Q_2$.

Now we claim that $\bigcup_{a \in U^M} \alpha_a$ is an isomorphism
from $M$ to $N$.  It is well-defined since we noted that any $\alpha_a$ and $\alpha_b$
agree on their common domain which is a subset of $Q_2$ and the union maps all of $M$ to all of $N$.

The
$Q_1,Q_2, < , P$ are clearly preserved. Finally, this is a
$\tau'$-isomorphism because each atomic $\tau'$-formula $\hat
R(\cdot,\vec{c})$ holds on the domain of some $\alpha_a$.

Moreover, note that if $M$ is a model of $\psi$ so that
all the $D^M_a$ are {\em maximal} $(\kappa,\lambda)$-models of $\chi_{\theta}$ then $(Q_2^M,<^M)$ is $\lambda^+$-like. So $|M|
\le\max(\kappa,\lambda^+)$ and there is a model in which that maximum is attained.
Now when $\lambda = \kappa$ there is a maximal $\tau'$-model  $M$  of $\psi$ with size $\kappa^+$ and when
$\lambda<\kappa$, $M$ is a maximal model of size $\kappa$; in both cases, $Q_2^M$ has size $\lambda^+$.
 $\qed_{\ref{getcompnew}}$

Note that Theorem~\ref{getcompnew} is a trivial corollary of
Theorem~\ref{Thm:ManyMaximal} if the answer to the following question
is positive.  But after considerable effort trying to modify the construction of \cite{Knightex}, the question seems to be harder
than
Theorem~\ref{getcompnew}.

\begin{question} Is there is a complete sentence of
$\lomegaone$ that has a $(\kappa^+,\kappa)$-model in every
cardinality?  More strongly, is there such a first order
$\aleph_0$-categorical theory?
\end{question}

Particular examples of homogeneously characterizable cardinals are given by \cite{BaumgartnersHanfNumber},
\cite{Hjorthchar},
\cite{Soul1}, \cite{Souldatoscharpow}, \cite{SouldatosCharacterizableCardinals}, \cite{BKL}.

\begin{Fact}[Theorem 4.29, \cite{Souldatoscharpow}] If $\aleph_\alpha$ is a characterizable cardinal, then
$2^{\aleph_{\alpha+\beta}}$ is homogeneously characterizable, for all $0 < \beta < \omega_1$.
\end{Fact}

\begin{Fact} If $\kappa$ is homogeneously
characterizable, then
so is each\footnote{1) Baumgartner; see also Theorem 3.4 of
    \cite{Souldatoscharpow}; 2) Theorem
    3.6,\cite{SouldatosCharacterizableCardinals};  3)Corollary 5.6, \cite{Soul1}.} of the following.

\begin{enumerate}
\item  $2^\kappa$ ;
\item  $\kappa^\omega$;
\item \label{SoulPow}
    $\kappa^{\aleph_\alpha}$, for all countable ordinals
    $\alpha$.
\end{enumerate}
 \end{Fact}

Finally a result of slightly different character; we note a direct
proof for each $n$ of a sentence $\phi_n$ that homogeneously
characterize $\aleph_n$ ($n> 0$) and has $(\aleph_n,\aleph_k)$ models
for $k\le n$.

\begin{Fact}[\cite{BKL}]\label{Thm:BKLExamples} For each $n\in\omega$, there is a complete $\lomegaone$-sentence
$\phi_n$ such that
\begin{itemize}
 \item $\phi_n$ homogeneously characterizes $\aleph_n$ with
     $(\phi_n,P)$ receptive; and
 \item for each $k\le n$, there is a maximal model $N_k$ of $\phi_n$ of type $(\aleph_n,\aleph_k)$.
\end{itemize}
\end{Fact}

Since in this last example, the complete sentence\footnote{The proof
that  these $(\aleph_n,\aleph_k)$ models exist  requires the use of
both frugal amalgamation   and an amalgamation which allows
identification. We say a class has frugal amalgamation
 if for every amalgamation triple $A, B,C$ there is an amalgam on
 the union of the domains with no identifications. See
 \cite{BKS}.}
has maximal models of type $(\aleph_n,\aleph_k)$, for all $k \leq n$
there is no need to appeal to Theorem \ref{Thm:ManyMaximal} for an
intermediate sentence.

\section{Maximal models in $\kappa$ and $\kappa^\omega$}\label{sec:ltoomega}
\numberwithin{Thm}{section}
\setcounter{Thm}{0}

Working similarly to Section \ref{template} we construct a complete
$\lomegaone$-sentence that admits maximal models in $\kappa$ and
$\kappa^\omega$, and has no larger models.  But we must define a
sentence that transfers from characterizing $\kappa$ to characterizing
$\kappa^{\omega}$ rather than to $\kappa^+$.

Although proved earlier (\cite{SouldatosCharacterizableCardinals}),
the following result can be viewed as an extension of the argument
for Theorem~\ref{getcompnew}.  We first have to replace well-known
fact that ${\rm Th} (Q,<)$ is first order $\aleph_0$-categorical by a
proof that the tree $\lambda^{<\omega}$ along with a set of dense
paths can be axiomatized in $\lomegaone$.  Then we extend
the trick illustrated in Theorem~\ref{getcompnew} to bound the number
of successors of each node in the tree by $\kappa$ and thus the
number of paths by $\kappa^\omega$. The detailed axiomatization of a
structure with these properties, but in a different
 vocabulary,
 by a complete sentence of
$\lomegaone$ and the proof that it characterizes $\kappa^\omega$
appears in \cite{SouldatosCharacterizableCardinals}.  The extension to show $\kappa^{\omega}$ is {\em homogeneously} characterized
requires the further analysis of Hjorth construction in the same paper.

For any vocabulary $\tau$ and $\tau$-predicate $R$ and $\tau$
structure $N$ we write $R^N$ for the interpretation of $R$ in $N$.

\begin{Thm}\label{ltoomega} Let $\phi$ be a complete
$\lomegaone(\tau)$-sentence (in vocabulary $\tau$) with a set of absolute indiscernibles $U$ that homogeneously characterizes
$\kappa$.
 Then there is
a complete $\lomegaone(\tau_2)$-sentence $\phi^*$ (in vocabulary
$\tau_2\supset\tau$) such that $\phi^*$
characterizes\footnote{$\phi^*$ does {\em not} homogeneously characterize $\kappa^{\omega}$.}
$\kappa^{\omega}$.

Moreover, let  $\mu$ be the least infinite cardinal such that
$\kappa\le \mu^\omega$. If $\mu>\aleph_0$, then  $\phi^*$  has
maximal models in $\kappa$ and $\kappa^\omega$, and no models larger
than $\kappa^\omega$. If $\mu=\aleph_0$, $\phi^*$ has maximal models
only in $2^{\aleph_0}$.
\end{Thm}

\begin{proof} We first show the structure with universe $M =
\omega^{< \omega} \cup \{f\in \omega^\omega\colon f {\rm \ is\
eventually\ constant}\}$ with the following relations has a Scott
sentence. Fix a vocabulary $\tau_1$ with unary predicates  $T,P, L_n$
for finite $n$, binary predicate $\unlhd$, and constant $0$
(none of which are in $\tau$).   The sentence $\phi_1$ in
$\lomegaone(\tau_1)$ describes the following structure on $M$: $T$
(tree) and $P$ (paths) partition the universe; $T$ denotes $\omega^{<
\omega}$ and $P$ denotes the eventually constant sequences.
$(M,\unlhd)$ is a tree of height $\omega+1$ ($\unlhd$ is a partial
order, with initial element $0$, such that the set of predecessors of
any element $v$ of $M$ is linearly ordered and includes $0$). An
element has finitely many predecessors if $v\in T$, while $P$
contains the elements of infinite height. But $v\in P$ implies every
$u\unlhd d$ has finite height. That is, $T^M = \bigcup_{n<\omega}
L_n$, where $L_n$ picks out the elements of `height' $n$. One easily
defines an `immediate extension' predicate $E(u,v)$ on $M^2$ (when $v
\not \in P$), which holds just if $u \unlhd v$ and $L_n(u)
\leftrightarrow L_{n+1}(v)$. Note that for any $v \in M$, there is a
unique definable restriction $v \restriction n$ (for any $n$ not
greater than the height of $v$).

Include in $\phi_1$ the crucial axioms for $\tau_1$-categoricity:

\begin{enumerate}
\item Each $v$ with finite height has infinitely many immediate
    extensions.
\item Each $v$ with finite height
    has infinitely many extensions in $P$.
\end{enumerate}

We first prove $\phi_1$ is a Scott sentence for the $\tau_1$-structure $(M,\unlhd,T,P, (L_n)_{n\in\omega},0)$.
 We construct a back-and-forth system between arbitrary models $M$ and $N$ of
$\phi_1$. Suppose  $A$ and $B$ are finite subsets of $M$ and $N$
respectively, and $\alpha\colon A \approx B$. Take any $c \in M\setminus A$.

If $c\unlhd a$ for some $a\in A$, the extension is easy.
If not, there exists a unique $a_c \in A$, maximal with $a_c\unlhd c$ and apply
axiom 1 or 2 depending whether $c \in T$ or $P$.

This completes the first step in the argument. Without loss of
generality we may replace $\phi$ by the $\chi_\phi$ from
Theorem~\ref{Thm:ManyMaximal} that has {\em maximal}
$(\kappa,\lambda)$ models for each $\lambda\leq \kappa$.

Now we use a slightly more complicated version of the strategy
for Theorem~\ref{getcompnew}. Form $\tau_2$ by adding a binary symbol
$D(\cdot,\cdot)$ to $\tau_1$ and an $n+1$-ary predicate $Q(x,\cdot)$
for each $n$-ary $\tau$-predicate $Q(\cdot)$.

Let $\phi^*$ be the conjunction of $\phi_1$ with the assertions that for $u\neq v \in T$
the sets $D(u,\cdot)$ and $D(v,\cdot)$ are disjoint and they are also disjoint from $T$ and $P$.

Require further that for each $u\in
V$, the set $D(u,\cdot)$ (under the relations $Q(u,\cdot)$) is a
model of $\phi$ and that the set $R(u,\cdot)$ of the
immediate successors of $u$ is also the set $U(u,\cdot)$ of absolute
indiscernibles of the model $D(u,\cdot)$ of $\phi$.  Since $\phi$
homogeneously characterizes $\kappa$, if $N \models \phi^*$,
$|R^N(u,\cdot)| \leq \kappa$.

To see that any countable
models of $M, N$ of $\phi^*$ are isomorphic, note first that we
already showed their $\tau_1$-reducts are isomorphic. The extension
to a $\tau_2$-isomorphism uses the absolute indiscernibility of $\{u:
u \unlhd v\}$ in $D(v,\cdot)$ as in Theorem~\ref{Thm:ManyMaximal}.

 If
for every $u\in V$, the set $D(u,\cdot)$ is a maximal model of
$\phi^*$ of type $(\kappa,\lambda)$, then the resulting tree is
$\lambda$-splitting and there is an associated maximal model of
$\phi^*$ of size $\max\{\kappa,\lambda^\omega\}$.

Take $\mu$ to be the least infinite cardinal such that $\kappa\le
\mu^\omega$. Thus, $\phi^*_\kappa$ has a maximal model of size $\mu^\omega$. Moreover, for any $\lambda$ with
$\mu\le\lambda\le\kappa$ by cardinal arithmetic
$\mu^\omega\le\lambda^\omega\le\kappa^\omega\le
(\mu^\omega)^\omega=\mu^\omega$. Also
for any $\lambda<\mu$, $\lambda^\omega<\kappa$ and $\phi^*$ has a maximal model of size $\max\{\kappa,\lambda^\omega\}=\kappa$.
Note that the model is maximal if each $D(\cdot,\cdot)$ is a maximal $(\kappa,\lambda)$-model and {\em each} path through
resulting tree on $\lambda^{<\omega}$ is realized.

For the last claim, if $\mu = \aleph_0$, then the only possible trees are on
$\aleph_0^{<\omega}$ and they must be $\aleph_0$-splitting. So there is
 a maximal model of $\phi^*$ of size
$\max\{\kappa,\aleph_0^\omega\}=2^{\aleph_0}$,
and every model of size less than $2^{\aleph_0}$ is not maximal.
\end{proof}

An easy application of Shoenfield's absoluteness theorem proves that for a countable vocabulary and for a sentence $\phi
\in \lomegaone$ the
existence of a countable maximal model is absolute. While the existence of a model in $\aleph_1$ is also absolute (For instance,
apply Keisler's
completeness theorem for $L(Q)$.), absoluteness fails for the
existence of a maximal model of size $\aleph_1$.

\begin{Cor}\label{nonabsmax} ``Existence of a maximal model in $\aleph_1$''
is not an absolute notion for models of ZFC. More precisely, there
exist two transitive models of ZFC, $V\subset W$, $\phi\in \lomegaone^V$, both $V$ and $W$ satisfy that ``$\phi$ has
models in $\aleph_1$", where $\aleph_1$ is interpreted in the corresponding model, and $V\models$ ``$\phi$ has a
maximal model in $\aleph_1$", while $W\models$ ``$\phi$ does not have a maximal model in $\aleph_1$".
\end{Cor}
\begin{proof}
 Let $\kappa$ equal $\aleph_1$. By \ref{Thm:BKLExamples}, $\kappa$ is homogeneously characterizable. Let $\phi^*$ be the complete
sentence given by Theorem~\ref{ltoomega}. Let $V$ be a model of CH
and $W$ an extension of $V$ in which CH fails. In $V$, $\phi^*$ has
maximal models only in $(2^{\aleph_0})^V=\aleph_1^V$. In $W$,
$\phi^*$ has maximal models only in $(2^{\aleph_0})^W>\aleph_1^W$.
\end{proof}

Replacing the construction that characterizes $\kappa^\omega$ from \cite{SouldatosCharacterizableCardinals} with the
construction that characterized $\kappa^{\aleph_\alpha}$, $\alpha<\omega_1$, from \cite{Soul1} (cf. Theorem
\ref{SoulPow}) one can prove the following theorem.

\begin{Thm}\label{kkalpha} Assume $\alpha<\omega_1$, $2^{\aleph_\alpha}<\kappa<\kappa^{\aleph_\alpha}$ and there is
a sentence
$\phi_\kappa$ that
homogeneously characterizes $\kappa$.
Then there is a
complete sentence $\phi^*_\kappa$ that has maximal models in $\kappa$ and $\kappa^{\aleph_\alpha}$, and no models larger than
$\kappa^{\aleph_\alpha}$.
\end{Thm}

\section{Consistency of Maximal models in  many cardinalities}\label{ManyCardinalities}
\numberwithin{Thm}{section}
\setcounter{Thm}{0}

In this section we construct a complete $\lomegaone$-sentence that
consistently admits maximal models in  many cardinalities.
We first give an easy argument to find maximal models in $\kappa$ and
$2^\kappa$ when $\kappa$ is homogeneously characterized.

In \cite{BaumgartnersHanfNumber}, Baumgartner
used independent families of sets to prove that if $\kappa$ is homogeneously characterizable, then
the same is true for $2^{\kappa}$. A similar result is  Theorem 4.29
of \cite{Souldatoscharpow} where the assumption of homogeneously characterizability of $\kappa$ is relaxed to a
$\kappa$ being characterized by a linear order.
Given the machinery of homogeneously characterizable cardinals and
mergers, our transfer theorem \ref{powerset} has a rather elementary proof.

\begin{Thm}\label{powerset} Suppose that $\phi$ is a complete $\lomegaone$-sentence  that
homogeneously
characterizes $\kappa$ with absolute indiscernibles in the predicate
$P$ and $\phi$ has no maximal models below $\kappa$. Then there is a complete $\lomegaone$-sentence $\phi_\kappa$
that characterizes $2^\kappa$.

Furthermore  for every $\lambda\le\kappa$, there exists a maximal
model of $\phi_\kappa$ of size $\max\{\kappa,2^\lambda\}$ and every maximal model has one of these cardinalities.
\end{Thm}

\begin{proof} By Theorem~\ref{Thm:ManyMaximal}, we can assume $\phi$ has
maximal models of type $(\kappa,\lambda)$ with absolute
indiscernibles $B$, for all $\lambda\le\kappa$.

Let $T$ be the $\aleph_0$-categorical first order theory saying $U$
and $V$ are disjoint infinite sets and $E$  is extensional so that
$E(v,\cdot)$ defines a family of subsets $X_v$ of $U$. Requiring that
every finite Boolean combination of the $X_v$ is non-empty (and
dually for the $Y_u = \{v\colon v \in u\}$) gives an
$\aleph_0$-categorical theory such that for every model $M$, $|V^M|
\leq 2^{|U|}$.

Merge $\mu =\bigwedge T$ with the complete sentence $\phi$ from
Theorem~\ref{Thm:ManyMaximal} identifying $U$ with $B$. Let
$\phi_\kappa=\chi_{\psi,B,\mu,U}$. By Fact \ref{mergeprop} (1), $\phi_\kappa$ is
a complete sentence of $\tau' = \tau_\psi \cup \{U,V,E\}$.

Now  $M$,
a maximal model of $\phi$ with type $(\kappa,\lambda)$, yields a
maximal model of $\phi_\kappa$ with cardinality $\max\{\kappa,2^\lambda\}$.
There can be no other maximal models as if $(M,N)$ is a maximal model of the merger $\phi_\kappa$  then $M$ is maximal and  if $|U^M| =|B^N|=
\lambda$, then $|V^N|$ must be $2^\lambda$.
\end{proof}

Exactly what this says about the cardinality of maximal models
depends on the cardinal arithmetic. We just give some sample
applications of Theorem~\ref{powerset} with various choices
of the $\lambda$ and of the set theoretic hypotheses.

We describe below some ways to arrange the values the powerset function assumes  on the interval $[\mu,\kappa]$ to illustrate the
effect of the next theorem.

\begin{Cor}\label{maxmodels} Assume $\kappa$ is a homogeneously characterizable cardinal and the characterizing sentence has no
maximal models below $\kappa$. Let $\mu$ be the least
cardinal such that $2^{\mu}\ge\kappa$.
Then there is a complete $\lomegaone$-sentence $\phi_\kappa$ with maximal models in {\em exactly} the cardinalities
$\kappa$ and
$2^\lambda$, for all $\mu\le\lambda\le\kappa$.
\end{Cor}

The difficulty is that it is impossible to specify in ZFC the equalities/inequalities among the $2^{\lambda}$'s.  In ZFC we
cannot specify  them as $\aleph$'s.
But, using Easton's theorem we can establish a number of possibilities.

%

\begin{Cor}\label{abetaseq} Let $\mu=\aleph_1$, $\kappa=2^{\aleph_1}$ and $\phi_\kappa$ be from
Corollary~\ref{maxmodels}. If  $\Gamma =(\alpha_i|i< \alpha_0)$
is an increasing sequence of ordinals
and $\cf(\aleph_{\alpha_i})>\aleph_{i+1}$, then  there is a $V^{\Gamma} \models ZFC$ such that $\phi_\kappa$ has maximal models in
exactly the cardinalities
$(\aleph_{\alpha_i}|i< \alpha_0)$ along with the values of the $2^{\aleph_\gamma}$ where $\gamma < \alpha_0$ and $\gamma$ is
a limit ordinal.

\begin{proof}
First  we apply Corollary~\ref{maxmodels} with $\mu=\aleph_1$ and $\kappa=2^{\aleph_1}$. We need the fact that
$\aleph_1$ is homogeneously characterizable, but this follows from \ref{Thm:BKLExamples}, and clearly a complete sentence
characterizing $\aleph_1$ can
have no maximal countable model. Then apply \ref{maxmodels}
The resulting sentence $\phi_\kappa$ from
Corollary~\ref{maxmodels}  has maximal models in all cardinalities
$2^\lambda$, for all $\aleph_1\le\lambda\le 2^{\aleph_1}$. Notice that $\phi_\kappa$ depends only on $\kappa$ and not on the
choice of $\Gamma$.

Next, we create a model $V^{\Gamma}$ of ZFC where  the set $\{2^\lambda|\aleph_1\le\lambda\le 2^{\aleph_1};\;\lambda \text{ a
successor}\}$ equals the set $\{\aleph_{\alpha_i}|\alpha_i \in \Gamma\}$, which proves the statement.
  We describe the cardinal arithmetic requirements
on $V^{\Gamma}$ carefully. Using Easton forcing,
 we ensure first that $2^{\aleph_1}$ equals $\aleph_{\alpha_0}$.
So, $\{2^\lambda|\aleph_1\le\lambda\le 2^{\aleph_1};\;\lambda \text{ a
successor}\}=\{2^{\aleph_{i+1}}|i<\alpha_0\}$.
Then using   the assumption on $\cf(\aleph_{\alpha_i})$,
Easton guarantees as well that in $V^{\Gamma}$,
$2^{\aleph_{i+1}}=\aleph_{\alpha_i}$, for all $i<\alpha_0$. So $\Gamma$ indexes a part of the range of the function giving the cardinality of power sets.

\end{proof}
\end{Cor}
We know a bit more.

\begin{enumerate}
\item
 If $\alpha_0>\omega$, the complete sentence given by Corollary \ref{abetaseq} will have maximal models in
other cardinalities than
$(\aleph_{\alpha_{i}}|i<\alpha_0)$.
 For instance,
for those $i$ where $\lambda =\aleph_i$ is singular,
Easton's theorem does not control the $\aleph$-index of $2^\lambda$, although we know there is a maximal model in that
cardinality.

 \item  Although the sentence $\phi_\kappa$ given by Corollary \ref{abetaseq} has maximal models in cardinalities that are
bounded by
$2^{2^{\aleph_1}}$, the same idea can be applied to other characterizable cardinals. However, since characterizable cardinals
are bounded
by $\beth_{\omega_1}$, the cardinalities where the maximal models occur are also
bounded by $\beth_{\omega_1}$.
\item The complete sentences given by Corollaries \ref{maxmodels} and \ref{abetaseq}  do not have arbitrarily large
models.
\end{enumerate}

\begin{Cor} For complete $\lomegaone$-sentences the number of cardinalities where maximal models occur is not absolute.
\end{Cor}

\section{Conclusion}\label{Sec:Conclusion}

The existence of maximal models in several cardinalities suggests the following strengthening of earlier question concerning the
number of models in a cardinal that is characterized.

\begin{question}\label{OpenQuestion1} Is there a complete $\lomegaone$-sentence $\phi$ which has at least one maximal
model in an uncountable cardinal $\kappa$, but less than $2^\kappa$
many  models of cardinality $\kappa$?
\end{question}

In particular, a negative answer to Open Question \ref{OpenQuestion1} implies a negative answer to the following Open
Question \ref{OpenQuestion2}, which was asked in \cite{BKL} and which in return relates to old conjectures
of S. Shelah.

\begin{question}[\cite{BKL}]\label{OpenQuestion2} Is there a complete $\lomegaone$-sentence which characterizes  an uncountable
cardinal
$\kappa$ and it has less than $2^\kappa$ many models in cardinality $\kappa$?
\end{question}
%

All the examples in this paper have maximal models in some cardinalities and using set-theory we can identify the maximality
cardinals in the $\aleph$-hierarchy. Our examples can not be used to settle whether the statement ``$\phi$ has a maximal model''
is absolute. We noticed already in the comments preceding Corollary \ref{nonabsmax} that existence of a countable maximal model
and existence of an uncountable model are absolute notions. So, it is necessary that a proposed counterexample will
consistently have a maximal model in an  uncountable cardinality. By Lemma 5.8 of \cite{BaldwinAmalgamation}, the property that
an $\lomegaone$-sentence has arbitrarily large models is absolute. This further implies that the proposed counterexample will
have arbitrarily large models in \emph{all} models of ZFC.

\begin{question}
Given an $\lomegaone$-sentence $\phi$, is the following statement absolute for transitive models of ZFC?
``$\phi$ has a maximal model in an uncountable cardinality''.

More precisely, do they exist two transitive models of ZFC, $V\subset W$, $\phi\in \lomegaone^V$, both $V$ and $W$ satisfy
that ``$\phi$ has arbitrarily large models'', and $V,W$ disagree on the statement
``$\phi$ has a maximal model in an uncountable cardinality''?
\end{question}

Finally, we want to stress the differences in techniques of this paper from \cite{BKSoul}. The main idea behind
\cite{BKSoul} is certain combinatorial properties of bipartite graphs.
Here the main construction is a refinement of old ideas,
e.g. the characterization of $\kappa^+$ by a $\kappa^+$-like linear order in Section \ref{template} and the characterization of
$\kappa^\omega$ using results from \cite{SouldatosCharacterizableCardinals} in Section \ref{sec:ltoomega}, combined with repeated
use of sets of absolute indiscernibles.
All the examples presented here are \emph{complete} sentences with maximal models in more than one cardinality,
which \emph{do not} have arbitrarily large models. In \cite{BKSoul} the examples are \emph{incomplete} sentences with
maximal
models in more than one cardinality, which \emph{do} have arbitrary large models.

{\bf Note:} Stimulated by early versions of this paper, Baldwin and Shelah began
the paper `The Hanf number for extendability and related phenomena'.
They construct (under mild set theoretic hypotheses which are
expected to be eliminated) a complete sentence of $\lomegaone$ with
maximal models arbitrarily high below the first measurable. Note that
every model above the first measurable has a proper
$\lomegaone$-elementary extension.  In contrast to this result the
method discussed in the last paragraph seem to be limited to
counterexamples below $\beth_{\omega_1}$. Can one find a sentence
$\phi$ with maximal models bounded somewhere between these bounds? If
not, can one explain why there is such an immense gap? Under ZFC +
"there exists a measurable cardinal", no compete sentence of
$\lomegaone$ has arbitrarily large maximal models. Under ZFC + `no
measurable cardinals', our only example with a maximal model of
cardinality beyond $\beth_{\omega_1}$ has arbitrarily large maximal
models. Is it always true that under ZFC + ``there are no measurable
cardinals'', if there is a maximal model of cardinality at least
$\beth_{\omega_1}$, then there are arbitrarily large maximal models.
Does this make the Hanf number for the existence of a maximal model
(with no measurable) $\beth_{\omega_1}$ or can more counterexamples
be constructed?

\end{document}